\documentclass[11pt]{article}%
\usepackage[utf8]{inputenc}
\usepackage[T1]{fontenc}
\usepackage[english]{babel}
\usepackage{indentfirst}
\usepackage{latexsym}
\usepackage[colorlinks=true, bookmarksnumbered=true, bookmarksopen=true,
bookmarksopenlevel=3, pdfstartview=FitH, linkcolor=cyan, pdfmenubar=true,
pdftoolbar=true, bookmarks=true,citecolor=cyan, urlcolor=magenta,
filecolor=cyan,menucolor=black,plainpages=false,pdfpagelabels]{hyperref}
\usepackage[lmargin=2cm,tmargin=2cm,bmargin=2cm,rmargin=2cm]{geometry}
\usepackage{amsbsy, amsmath,amsfonts,amssymb,amsthm}
\usepackage{times}
\usepackage{graphicx}
\usepackage{amsthm}
\usepackage{amsmath}
\usepackage{amsfonts}
\usepackage{amssymb}%
\newtheorem{theorem}{Theorem}[section]

\newtheorem{proposition}[theorem]{Proposition}
\newtheorem{corollary}[theorem]{Corollary}
\newtheorem{lemma}[theorem]{Lemma}
\newtheorem{remark}[theorem]{Remark}

\newtheorem*{theorem*}{Theorem}
\newtheorem*{question*}{Question}
\numberwithin{equation}{section}
\begin{document}

\title{On the reflexivity of $\mathcal{P}_{w}(^{n}E;F)$}
\author{{{Sergio A. Pérez}{\thanks{S.Pérez was supported by CAPES and CNPq, Brazil. }}}\\{\small  IMECC, UNICAMP} \\{\small {Rua S\'{e}rgio Buarque de Holanda, 651, CEP 13083-859, Campinas-SP,
Brazil.}}\\{\small \texttt{Email:Sergio.2060@hotmail.com}}\\\vspace{-0.2cm}}
\date{}
\maketitle

\begin{abstract}

In this paper we prove that if $E$ and $F$ are reflexive Banach spaces and $G$ is a closed linear subspace of  the space $\mathcal{P}_{w}(^{n}E;F)$ of all $n$-homogeneous polynomials from  $E$ to $F$ which  are weakly continuous on bounded sets, then $G$ is either reflexive or non-isomorphic to a dual space. This result generalizes \cite[Theorem 2]{FEDER} and gives the solution to a problem posed by Feder \cite[Problem 1]{FED}.

%Let $\mathcal{P}_{w} (^{n}E; F)$ denote the subspace of all $P\in \mathcal{P}(^{n}E; F)$ which are weakly continuous on bounded sets. We show that if
%$E$ and $F$ be reflexive Banach spaces and $G$ a closed linear subspace of $\mathcal{P}_{w}(^{n}E;F)$. Then $G$ is either reflexive or non-isomorphic to a dual space. This result answer a question put by Feder in \cite{FED}, moreover, it gives a generalization of \cite[Theorem 2]{FEDER}.
{\small \bigskip\noindent\textbf{AMS MSC:}46B10, 46G25 }

{\small \medskip\noindent\textbf{Keywords:} Banach space, linear operator, compact operator, homogeneous polynomial, dual space, reflexive space.}

\end{abstract}

\section{Introduction}

An important result of  Feder \cite{FED} states that if $E$ and $F$ are reflexive Banach spaces such that $F$ or $E^{\prime}$ is a subspace of a Banach space with an unconditional basis, then the space $\mathcal{L}_{K}(E;F)$ of all compact linear operators from $E$ to $F$ is either reflexive or
non-isomorphic to a dual space. In \cite{FEDER}, Feder and Saphar proved that if $E$ and $F$ are reflexive Banach spaces and $G$ is a closed linear
subspace of $\mathcal{L}_{K}(E;F)$ which contains the space $\mathcal{R}(E,F)$ of all finite rank linear operators from $E$ to $F$, then $G$ is either reflexive or non-isomorphic to a dual space. But the following question posed in \cite{FED} remains open:
\begin{question*} Let $E$ and $F$ be reflexive Banach spaces. Is $\mathcal{L}_{K}(E;F)$ either reflexive or
non-isomorphic to a dual space?
\end{question*}
In this paper, we obtain a positive answer for the previous question. In fact, we prove the following more general result:
\begin{theorem*} Let $E$ and $F$ be reflexive Banach spaces and $G$ be a closed linear subspace of $\mathcal{P}_{w}(^{n}E;F)$. Then $G$ is either reflexive or non-isomorphic to a dual space.
\end{theorem*}
The answer of the aforementioned question is obtained just considering $n=1$ in the previous theorem. As other consequences of this result we also obtain two conditions, one that ensures that $\mathcal{P}_{w}(^{n}E;F)$ is non-isomorphic to a dual space and other such that $\mathcal{P}_{w}(^{n}E;E)$ is non-isomorphic to a dual space (see Corollaries \ref{thm:(Teorema 15)} and \ref{thm:(Teorema 16)} ). We also obtain a generalization of Boyd and Ryan \cite[Theorem 21]{BOYD}.

Throughout this paper $E$ and $F$ denote Banach spaces over $ \mathbb{K}$, where $ \mathbb{K}$ is $ \mathbb{R}$ or  $\mathbb{C}$,  $E^{\prime}$
denotes the dual of $E$ and $B_{E}=\{x\in E: \|x\|\leq1\}$. We say that $E$ is a \emph{dual space} if there exists a Banach space $X$ such that $X^{\prime}=E$. Let $J_E:E\rightarrow E^{\prime\prime}$ denotes the canonical injection from $E$ into $E^{\prime\prime}$. The space of all bounded linear operators from $E$ to $F$ is represented by $ \mathcal{L}(E;F)$ and $\mathcal{P}(^{n}E; F)$ denotes the Banach space of all continuous $n$-homogeneous polynomials from $E$ into $F$ with its usual sup norm. We omit $F$ when $F = \mathbb{K}$. Let
$\mathcal{P}_{f}(^{n}E; F)$ denotes the subspace of $\mathcal{P}(^{n}E; F)$ generated by all polynomials of the form $P(x)= (\phi(x))^{n} b$,
with $\phi\in E^{\prime}$ and $b\in F$. We denote by $\mathcal{P}_{A}(^{n}E; F)$ the closure of $\mathcal{P}_{f}(^{n}E; F)$ with the norm topology and $\mathcal{P}_{w}(^{n}E; F)$ denotes the subspace of $\mathcal{P}(^{n}E; F)$ formed by all $P$
which are \emph{weakly continuous on bounded sets}, that is the restriction $P|_{B}: B \rightarrow F $ is continuous for each bounded set $B\subset E$,
when $B$ and $F$ are endowed with the weak topology and the norm topology, respectively. The subspace $\mathcal{P}_{K}(^{n}E; F)$ of $\mathcal{P}(^{n}E; F)$  is formed by all polynomials that send bounded sets onto relatively compact sets.
It is well-known that
$$ P_{w}(^{n}E; F)\subset \mathcal{P}_{K}(^{n}E; F)\subset \mathcal{P}(^{n}E; F)$$
 and $P_{w}(^{n}E; F)=\mathcal{L}_{K}(E;F)$ when $n=1$.
For $T\in\mathcal{L}(E;F)$ we denote by $T^{\prime}\in\mathcal{L}(F^{\prime};E^{\prime})$ the adjoint operator of $T$.
Finally, let us recall that $E$ has the \emph{compact approximation property} (CAP in short) if given a compact set $C \subset E$
and $\epsilon > 0$, there is $T\in \mathcal{L_{K}}(E;E)$ such that $\|Tx-x\| < \epsilon$
for every $x\in C$.

\smallskip

\section{The main result}

To prove the main result, we need the following lemma, which is a special case of \cite[Corollary 5]{Gon}.

\begin{lemma}\label{thm:(Teorema 1)}Let $E$ and $F$ be reflexive Banach spaces and $G$ be a closed linear subspace of $\mathcal{P}_{w}(^{n}E;F)$. Let $P_{m}, P\in G$ for each $m\in\mathbb{N}$. Then $\lim\limits_{m \rightarrow \infty}P_{m}=P$ weakly in $G$ if and only if $\lim\limits_{m \rightarrow \infty}y^{\prime}(P_{m}(x))=y^{\prime}(P(x))$ for every $x\in E$ and every $y^{\prime}\in F^{\prime}$.
\end{lemma}

\begin{theorem}\label{thm:(Teorema 13)}
Let $E$ and $F$ be reflexive Banach spaces and $G$ be a closed linear subspace of $\mathcal{P}_{w}(^{n}E;F)$. Then $G$ is either reflexive or non-isomorphic to a dual space.
\end{theorem}

\begin{proof}
Suppose that $G$ is isomorphic to the conjugate of a Banach space $X$. Let $\varphi:X^{\prime}\rightarrow G$ be an isomorphism. To show that $G$ is reflexive, we need to prove that every sequence in $B_{G}$ has a weakly convergent subsequence. Consider $(P_{m})$ in $B_{G}$. Since $B_{G^{\prime\prime}}$ is
$\sigma(G^{\prime\prime},G^{\prime})$-compact there exist a subsequence $(J_G(P_{m_{k}}))$ of $(J_G(P_{m}))$ and $\theta \in G^{\prime\prime}$ such that $\lim\limits_{k \rightarrow \infty}J_G(P_{m_{k}})=\theta$ in the $\sigma(G^{\prime\prime},G^{\prime})$-topology.
For every $y^{\prime}\in F^{\prime}$ and $x\in E$, consider the linear functional $$\psi_{y^{\prime},x}: P\in G\rightarrow y^{\prime}(P(x))\in \mathbb{K}.$$
Since $\psi_{y^{\prime},x}\in G^{\prime}$ we have that $$\lim\limits_{k \rightarrow \infty}\big<J_G(P_{m_{k}}),\psi_{y^{\prime},x}\big>=\lim\limits_{k \rightarrow \infty}y^{\prime}(P_{m_{k}}(x))=\theta(\psi_{y^{\prime},x})$$ for every $y^{\prime}\in F^{\prime}$ and $x\in E$. We want to prove that
$$\pi:\phi\in G^{\prime\prime}\rightarrow J_G\circ\varphi\circ J_X^{\prime}\circ (\varphi^{\prime\prime})^{-1}(\phi)\in J_G(G)$$
is a projection.Note that
\begin{align*}
\big<J_X^{\prime}\circ(\varphi^{\prime\prime})^{-1}(J_G(P)),z\big>=\big<(\varphi^{\prime\prime})^{-1}(J_G(P)),J_X(z)\big>=\big<(\varphi^{-1})^{\prime\prime}(J_G(P)),J_X(z)\big>=
\big<J_G(P),(\varphi^{-1})^{\prime}(J_X(z))\big>\\=\big<(\varphi^{-1})^{\prime}(J_X(z)),P\big>=\big<J_X(z),\varphi^{-1}(P)\big>=\big<\varphi^{-1}(P),z\big>
\end{align*}
for each $P\in G$ and $z\in X$. This implies that
$$J_X^{\prime}\circ (\varphi^{\prime\prime})^{-1}(J_G(P))=\varphi^{-1}(P)$$ and then
$\pi\circ J_G=J_G$. Thus $\pi$ is a projection and so $$G^{\prime\prime}=J_G(G)\oplus \ker(\pi).$$  Let $Q\in G$  and $\eta\in \ker(\pi)$ such that $\theta=J_G(Q)+\eta$. Since $\eta\in \ker(\pi)$ and $J_G\circ \varphi$ is injective, we have $J_{X}^{\prime}\circ(\varphi^{\prime\prime})^{-1}(\eta)=0$. On the other hand,

$$\eta(\psi_{y^{\prime},x})=\big<(\varphi^{\prime\prime})^{-1}(\eta),\varphi^{\prime}(\psi_{y^{\prime},x})\big>
=\big<J_{X}^{\prime\prime}(\varphi^{\prime}(\psi_{y^{\prime},x}),(\varphi^{\prime\prime})^{-1}(\eta)\big>
=\big<\varphi^{\prime}(\psi_{y^{\prime},x}),J_{X}^{\prime}\circ (\varphi^{\prime\prime})^{-1}(\eta)\big>=0.$$

Hence $$\lim_{k\rightarrow\infty}y^{\prime}(P_{m_{k}}(x))=\theta(\psi_{y^{\prime},x})=\big<J_{G}(Q),\psi_{y^{\prime},x}\big>+
\big<\eta,\psi_{y^{\prime},x}\big>=y^{\prime}(Q(x)),$$ for every $y^{\prime}\in F^{\prime}$ and $x\in E$. By Lemma \ref{thm:(Teorema 1)} it follows that $\lim\limits_{k \rightarrow \infty}P_{m_{k}}=Q$ weakly in $G$. This completes the proof.

\end{proof}

The next result is just Theorem \ref{thm:(Teorema 13)} with $n=1$. This is also the affirmative answer of \cite[Problem 1]{FED} and consequently a generalization of \cite[Theorem 2]{FEDER} and \cite[Theorem 5]{FED}.

\begin{corollary}\label{thm:(Teorema 14)}
Let $E$ and $F$ be reflexive Banach spaces and $G$ be a closed linear subspace of $\mathcal{L}_{K}(E;F)$. Then $G$ is reflexive or non-isomorphic to a dual space.
\end{corollary}

%\begin{remark}
%Corollary \ref{thm:(Teorema 14)} gives a generalization of \cite[Theorem 2]{FEDER} and \cite[Theorem 5]{FED},  resolving thus the problem $1$ put by Feder in \cite{FED}.
%\end{remark}

\begin{remark}
Note that Theorem \ref{thm:(Teorema 13)} does not work for $\mathcal{P}_{K}(^{n}E;F)$ instead of $\mathcal{P}_{w}(^{n}E;F)$. In fact, $\mathcal{P}_{K}(^{2}\ell_{2})=\mathcal{P}(^{2}\ell_{2})=\mathcal{L}(\ell_{2};\ell_{2} )$ is a dual space that is not reflexive.
\end{remark}

The next result is a generalization of \cite[Theorem 21]{BOYD}.

\begin{corollary}\label{thm:(Teorema 1000000)}
Let $E$ be a reflexive Banach space. Then $\mathcal{P}_{A}(^{n}E)$ is either reflexive or non-isomorphic to a dual space for every $n\in\mathbb{N}$.
\end{corollary}

\begin{corollary}\label{thm:(Teorema 100000000000)}
Let $E$ and $F$ be reflexive Banach spaces and $G$ be a closed linear subspace of $\mathcal{P}_{w}(^{n}E;F)$. If $\mathcal{P}(^{n}E;F)$ is isomorphic to $G$, then  $\mathcal{P}(^{n}E;F)$ is reflexive.
\end{corollary}
\begin{proof}
Since $\mathcal{P}(^{n}E;F)$ is a dual space, then the conclusion follows from Theorem \ref{thm:(Teorema 13)}.
\end{proof}

The next proposition is a particular case of \cite[Proposition 5.3]{AR}.

\begin{proposition}\label{thm:(Teorema 345)}
Let $E$ and $F$ be Banach spaces. Then $\mathcal{P}_{w}(^{n}E;F)$ is isomorphic to a closed subspace
of $\mathcal{P}_{w}(^{m}E;F)$ for every $m\geq n$.
\end{proposition}

\begin{proof}
To prove the proposition by induction on $n$ it suffices to prove that
 $\mathcal{P}_{w}(^{n}E; F)$ is isomorphic to a closed subspace of $\mathcal{P}(^{n+1}E; F)$.
 Choose $\varphi\in E^{\prime}$ such that $\varphi\neq 0$. Define $\rho: \mathcal{P}_{w}(^{n}E; F)\rightarrow \mathcal{P}_{w}(^{n+1}E; F)$ by
 $\rho(Q)(x)=\varphi(x)Q(x)$ for all $x\in E$. It is clear that $\rho$ is an injective linear operator. Therefore $\mathcal{P}_{w}(^{n}E; F)$ is isomorphic to $\rho(\mathcal{P}_{w}(^{n}E; F))\subset \mathcal{P}_{w}(^{n+1}E; F)$. This completes the proof.
\end{proof}

Finally we obtain the following results.

\begin{corollary}\label{thm:(Teorema 15)}
Let $E$ and $F$ be reflexive Banach spaces such that $E$ has the CAP. If $\mathcal{P}_{w}(^{n}E;F)\neq \mathcal{P}(^{n}E;F)$, then $\mathcal{P}_{w}(^{m}E;F)$ is not isomorphic to a dual space for every $m\geq n$.
\end{corollary}

\begin{proof}
By Theorem \ref{thm:(Teorema 13)} we only need to prove that $\mathcal{P}_{w}(^{m}E;F)$ is not reflexive for every $m\geq n$. By Proposition \ref{thm:(Teorema 345)}  we have that $\mathcal{P}_{w} (^{n}E;F)$ is isomorphic to a closed subspace of $\mathcal{P}_{w} (^{m}E;F)$ for every $m\geq n$.  If we assume that $\mathcal{P}_{w} (^{m}E;F)$ is reflexive for some $m\geq n$, then $\mathcal{P}_{w} (^{n}E;F)$ is also reflexive. By \cite[Corollary 4.4]{JI} we have that $\mathcal{P}_{w} (^{n}E;F)= \mathcal{P} (^{n}E;F)$, but this contradicts the hypothesis.
\end{proof}

\begin{corollary}\label{thm:(Teorema 16)}
Let $E$ be a reflexive infinite dimensional Banach space with the CAP. Then $\mathcal{P}_{w}(^{n}E;E)$ is non-isomorphic to a dual space for every $n\in \mathbb{N}$.
\end{corollary}

\begin{proof}
 By the Riesz Theorem  $\mathcal{L}_{K}(E;E)\neq \mathcal{L}(E;E)$. Now the result follows from Corollary \ref{thm:(Teorema 15)}.
\end{proof}

This paper is based on part of the author's doctoral thesis at the Universidade Estadual de Campinas. This research has been supported by CAPES and CNPq.
The author is grateful to his thesis advisor, Professor Jorge Mujica, for his advice and help. He is also grateful to Professor Vinícius
Fávaro for some helpful suggestions.

\end{document}